\documentclass[11pt,letterpaper]{amsart}

\usepackage[
colorlinks=true,              
linkcolor=blue,                
citecolor=red,               
urlcolor=cyan                  
]{hyperref}                        
\usepackage{amssymb}
\numberwithin{equation}{section}
\theoremstyle{plain}
\newtheorem{theorem}{Theorem}[section]
\newtheorem{prop}[theorem]{Proposition}

\newtheorem{lemma}[theorem]{Lemma}

\newtheorem{corollary}[theorem]{Corollary}

\theoremstyle{definition}

\newcommand{\Rmnum}[1]{\expandafter\@slowromancap\romannumeral #1@}

\newcommand{\Ric}{\operatorname{Ric}}
\newcommand{\ud}{\mathrm{d}}

\newcommand{\ms}{\mathbb{S}}

\allowdisplaybreaks
\newcommand{\R}{\mathbb{R}}
\newcommand{\Vol}{\operatorname{Vol}}
\newcommand{\dv}{\ud\mu_g}

\keywords{Yamabe problem,  $Q$-curvature, Paneitz operator}
\subjclass{53A30, 58J05}
\address{Mingxiang Li, Department  of Mathematics \& Institue of Mathematical Sciences, The Chinese University of Hong Kong}
\email{mingxiangli@cuhk.edu.hk}

\begin{document}
	\title[Positive Yamabe invariant and Paneitz operator]{On the  positivity of  Yamabe invariant and Paneitz operator}
	\author{Mingxiang Li}
	\date{}
	\maketitle
	
	\begin{abstract}Let $(M^n,g)$ be a smooth compact Riemannian manifold of dimension $n\ge 5$. We show that the existence of a conformal metric with positive $Q$-curvature $Q_g$ and positive scalar curvature $R_g$  is equivalent to the positivity of both the Yamabe invariant $Y(M^n,[g])$ and the Paneitz operator $P_g$. For $n=5$, this equivalence confirms a conjecture of Gursky-Hang-Lin (2016, IMRN). Furthermore, assuming $Y(M^n,[g])>0$, $Q_g\ge 0$, and $Q_g\not\equiv 0$, we prove that both $R_g$ and $P_g$ are positive which resolves  a problem  of Hang-Yang (2016, CPAM).  As a corollary, we show that the hypotheses of Gursky-Malchiodi (2015, JEMS) are equivalent to those of Hang-Yang (2016, CPAM).
	\end{abstract}
	
	\section{Introduction}
	
We begin by briefly recalling the development of the Yamabe problem, which has been a central topic in conformal geometry for several  decades. 
Let $(M^n,g)$ be a smooth compact  Riemannian manifold of  dimension $n\geq 3$.   
The Yamabe problem asks whether there exists a conformal metric with constant scalar curvature. This problem has been extensively studied and completely resolved through the works of Yamabe, Trudinger, Aubin, and Schoen (see \cite{Aubin, LP}).

The scalar curvature of the conformal metric is governed by a second-order operator, known as the conformal Laplacian, which is defined as
\begin{equation}\label{eq:conformal-Laplacian}
	L_gu:=-\frac{4(n-1)}{n-2}\Delta_gu+R_gu.
\end{equation}
where $\Delta_g$ is the Laplace-Beltrami operator and $R_g$  denotes the scalar curvature.
Consider a conformal metric $\tilde g=u^{\frac{4}{n-2}}g$ whose  scalar curvature is given by   $$R_{\tilde g}=u^{-\frac{n+2}{n-2}}L_gu.$$ 
By classical variational method,  to solve Yamabe problem, it is sufficient to  prove the existence of a positive minimizer for the Yamabe invariant
$$Y(M,[g]):=\inf_{u\in H^1(M)\setminus \{0\}}\frac{\int_MuL_gu\dv}{\left(\int_M|u|^{\frac{2n}{n-2}}\dv\right)^{\frac{n-2}{n}}.}$$
For convenience, we sometimes ignore the superscript $n$ of $M^n$ whenever no confusion arises.
For the standard sphere $(\ms^n,g_0)$, its Yamabe invariant is given by $n(n-1)|\ms^n|^{\frac{2}{n}}$ 
where $|\ms^n|$ denotes the volume of standard sphere.
The Aubin's criterion (see \cite{Aubin}) states that  if $Y(M,[g])<n(n-1)|\ms^n|^{\frac{2}{n}}$, then a positive minimizer exists.
 $Y(M,[g])>0$ is equivalent to the existence of a conformal metric with positive scalar curvature. Moreover, the positivity of  $Y(M,[g])$ is also equivalent to the positivity of the Green's  function $G_L$ of the conformal Laplacian $L_g$.

We now  turn our attention to $Q$-curvature. For brevity, we introduce  the  notation $
J_g:=\frac{R_g}{2(n-1)}$, denote Ricci tensor by $\Ric_g$ and 
denote the Schouten tensor by 
$
A_g:=\frac{1}{n-2}(\Ric_g-J_gg).
$
As a fourth order generalization of the conformal Laplacian \eqref{eq:conformal-Laplacian}, the Paneitz operator is defined as
\begin{equation}\label{Paneitz-operator}
	P_g\varphi
	=\Delta_g^2\varphi
	+\operatorname{div}_g\bigl(4A_g(\nabla\varphi,e_i)e_i
	-(n-2)J_g\nabla\varphi\bigr)
	+\frac{n-4}{2}Q_g\varphi
\end{equation}
where $\{e_i\}$ is a local orthonormal frame with respect to $g$ and $Q_g$ is Branson's $Q$-curvature  defined by 
\begin{equation}\label{Q-formula}
	Q_g=-\Delta_g J_g-2|A_g|_g^2+\frac{n}{2}J_g^2.
\end{equation}
Both \eqref{Paneitz-operator} and \eqref{Q-formula} hold for each dimension  $n\geq 3$. While the case $n=4$ is relatively well understood (see \cite{CY, DM, Gursky-CMP} for more discussion), the case $n=3$ is markedly different (see \cite{Hang-Yang note,HY-n=3}). Throughout this paper, we focus on the dimension  $n\geq 5$.

 When  $n\geq 5$, for a conformal metric $\tilde g=u^{\frac{4}{n-4}}g$, the Paneitz operator satisfies the conformal covariance property
\begin{equation*}
	P_{\widetilde g}\varphi
	=u^{-\frac{n+4}{n-4}}P_g(u\varphi).
\end{equation*}
In particular, the $Q$-curvature of such conformal metric $\tilde g$ satisfies
\begin{equation}\label{Q-def}
	Q_{\widetilde g}
	=\frac{2}{n-4}u^{-\frac{n+4}{n-4}}P_gu.
\end{equation}
Similar to the Yamabe invariant, it is natural to define a conformal invariant associated with the Paneitz operator as follows 
\begin{equation}\label{eq:Y_4-def}
	Y_4(M,[g]):=\inf_{u\in H^2(M)\setminus \{0\}}\frac{\int_MuP_gu\dv}{\left(\int_M|u|^{\frac{2n}{n-4}}\dv\right)^{\frac{n-4}{n}}}.
\end{equation}
Thanks to Sard's theorem, we may assume, without loss of generality, that 
$u$ is nonnegative by taking 
$|u|$ in the Yamabe invariant. However, the situation is very different when considering $Y_4(M,[g])$. Thus, Gursky, Hang, and Lin \cite{GHL} introduced two other conformal invariants.  The first is given by
$$Y_4^+(M,[g]):=\inf_{\substack{u\in C^\infty(M)\\ u>0}}\frac{\int_MuP_gu\dv}{\left(\int_Mu^{\frac{2n}{n-4}}\dv\right)^{\frac{n-4}{n}}}.$$
  When the Yamabe invariant $Y(M,[g])>0$, they further define
\begin{equation}\label{eq:Y4-star}
	Y_4^*(M,[g])
	=\frac{n-4}{2}
	\inf_{\substack{\widetilde g\in[g]\\ R_{\widetilde g}>0}}
	\frac{\displaystyle\int_M Q_{\widetilde g}\,d\mu_{\widetilde g}}
	{\Vol_{\widetilde g}(M)^{\frac{n-4}{n}}}.
\end{equation}
It follows immediately from the definitions that
\begin{equation}\label{Y_4-relation}
	Y_4(M,[g])\leq Y_4^+(M,[g])\leq Y_4^*(M,[g]).
\end{equation}

By the classical  variational method, one can find a conformal metric with constant $Q$-curvature provided the existence of a positive minimizer of the $Y_4(M,[g])$.
However, different from the conformal  Laplacian operator, the Paneitz operator generally lacks a strong maximum principle or the positivity of its Green's function. Assuming such properties hold a priori, there are many classical variational results concerning the Paneitz operator including  \cite{ER, Hebey-Robert, HR} and the references therein. For further discussion, we refer to the nice lecture notes of Robert \cite{Robert}.  Similar to Aubin's criterion, when $0<Y_4(M,[g])<Y_4(\ms^n,[g_0])$ and the strong maximum principle holds, one can find a conformal metric with constant $Q$-curvature (see  Proposition 4.1 and Theorem 5.2 in \cite{Robert}). In the case of locally conformally flat manifolds with  Poincar\'e exponent less than $(n-4)/2$, Qing and Raske \cite{QR} showed that the strong maximum principle holds and also established the existence of a conformal metric with constant $Q$-curvature.

Finding a natural geometric condition that  ensures the strong  maximum principle or the positivity of the Green's function for the  Paneitz operator was  a longstanding and difficult problem. The first breakthrough was  due to  Gursky and Malchiodi \cite{GM}, who showed that,  under the condition 
\begin{equation}\label{condition-of-GM}
	Q_g\geq 0\;\;  \mathrm{and} \;\; Q_g\not\equiv 0,\quad R_g\geq 0, 
\end{equation}
 the strong maximum principle for $P_g$ holds and then the Green's function $G_P$ for $P_g$ is positive.  Roughly speaking, they ingeniously applied the second-order strong  maximum principle twice,  based on the structure of $Q$-curvature and the Paneitz operator.
 Subsequently, Hang and Yang \cite{ Hang-Yang 15 IMRN} established the  positivity of the Green's function $G_{P}$  under the relatively relaxed condition
 \begin{equation}\label{condition-of-HY}
 	Q_g\geq 0\;\;  \mathrm{and} \;\; Q_g\not\equiv  0,\quad Y(M,[g])>0.
 \end{equation}
They established  a remarkable identity \eqref{eq:HY-identity} involving  $P_g(G_L^{\frac{n-4}{n-2}})$. For  the  standard sphere, $G_P$ is equal to $G_L^{\frac{n-4}{n-2}}$ up to a positive constant factor. The key idea  in \cite{Hang-Yang 15 IMRN} is to approximate  $G_P$ by $G_L^{\frac{n-4}{n-2}}$ and then obtain the identity \eqref{eq:HY-identity}.
Moreover, under the assumption \eqref{condition-of-HY},  Hang and  Yang \cite{Hang-Yang 15 IMRN} further proved that  $G_P\geq c(n)G_L^{\frac{n-4}{n-2}}$ for some positive constant $c(n)$ depending only on $n$.
 
Another key ingredient about $Q$-curvature  concerns  the positivity of the  Paneitz operator. We say that  $P_g$ is positive (denoted  simply by $P_g>0$) if, for all $u\in H^2(M)$, one has
\begin{equation}\label{positive P_g}
	\int_MuP_gu\dv\geq \lambda_1(g)\int_Mu^2\dv
\end{equation}
for some positive constant $\lambda_1(g)$. Using  the Sobolev embedding theorem and H\"older's inequality (see \cite{DHL}), $P_g>0$ is equivalent to $Y_4(M,[g])>0$. Under the assumption \eqref{condition-of-GM} with  $n\geq 5$, Gursky and Malchiodi \cite{GM}  (see Xu-Yang \cite{Xu-Yang} for $n\geq 6$) established the positivity of $P_g$ by employing   a  trick developed in  Gursky's remarkable  work \cite{Gursky-CMP} for the  four-dimensional case. Under the relaxed  assumption \eqref{condition-of-HY}, Hang and Yang \cite{Hang-Yang CPAM ngeq5}  raised the question whether    $P_g>0$ still holds.   This question is explicitly stated in \cite[Problem 3.14]{Hang-Yang note}. In this paper, we answer this question affirmatively.

Under the assumption \eqref{condition-of-GM}, Gursky and Malchiodi \cite{GM} first established the strong maximum principle and the positivity  $P_g>0$,  and then employed a nonlocal flow method to prove the existence of a conformal metric with constant  $Q$-curvature. Under condition \eqref{condition-of-HY}, Hang and Yang adopted a different strategy. They considered the extremal problem involving a new functional associated with the Paneitz operator:
$$\Theta_4(M,[g]):=\frac{2}{n-4}\sup_{\tilde g\in[g]}\frac{\int_MQ_{\tilde g}\ud\mu_{\tilde g}}{\left(\int_M|Q_{\tilde g}|^{\frac{2n}{n+4}}\ud\mu_{\tilde g}\right)^{\frac{n+4}{2n}}}.$$ Combining this with the positivity of the Green's function  established in \cite{Hang-Yang 15 IMRN}, they showed the existence of a conformal metric with constant 
 $Q$-curvature.  

It is well known that  the Yamabe invariant $Y(M,[g])>0$ is equivalent to the existence of conformal metric with positive scalar curvature. When $Q$-curvature is taken into account, it is natural to ask whether there exist some positive conformal invariant quantities  characterize the existence of a conformal metric with both positive $Q$-curvature and scalar curvature (see \cite [Problem 1.1]{GHL}). For $n\geq 6$, Gursky, Hang and Lin \cite{GHL} employed the continuity method to address this problem under the assumption
\begin{equation}\label{condition-of-GHL}
	Y_4^*(M,[g])>0,\quad Y(M,[g])>0.
\end{equation}
Moreover, they showed that all equalities in \eqref{Y_4-relation} hold in this situation. Due  to certain technical limitations, they were unable to deal with the case   $n=5$.  This has remained a long standing open problem ever since. In this paper, we employ the method developed by Hang and Yang in \cite{Hang-Yang note, Hang-Yang 15 IMRN} to fill this last gap. Our first result is the following.
\begin{theorem}\label{thm:main theorem}
	Let $(M^n,g)$ be a smooth  compact  Riemannian manifold of dimension $n\geq 5$.  Suppose that 
$Y(M,[g])>0$ and $	Y_4^*(M,[g])>0$. Then
there exists a conformal metric
	$\tilde g\in [g]$ satisfying
	$R_{\tilde g}>0$ and $Q_{\tilde g}>0.$
\end{theorem}

As noted earlier, Hang and Yang \cite{Hang-Yang note, Hang-Yang CPAM ngeq5} asked whether the condition \eqref{condition-of-HY}  implies  the positivity of $P_g$. Proving this directly appears to be difficult. In our previous  work  \cite{Li-Xu}, joint with Xu, we discovered that for any conformal metric on the standard sphere, the positivity of $Q$-curvature directly implies the positivity of the scalar curvature, where  we  employed  potential theory to represent the scalar curvature as a positive integral. Motivated by this observation and the integral representation, we prove that under the assumption \eqref{condition-of-HY}, the scalar curvature itself is indeed positive. Then, applying Proposition 2.3 of \cite{GM}, we obtain the positivity of the Paneitz operator $P_g$.  As a direct corollary, the conditions \eqref{condition-of-GM} and \eqref{condition-of-HY} are  equivalent.   We state our second result as follows.
\begin{theorem}\label{thm:positive Q imply positive P}
	Let $(M^n,g)$ be a smooth  compact  Riemannian manifold with dimension $n\geq 5$. Suppose that 
	$Y(M,[g])>0$, $	Q_g\geq 0$ and $Q_g\not\equiv 0$. Then, there holds
	$R_g>0$ and  $P_g>0.$
\end{theorem}

Combining Theorem \ref{thm:main theorem} and Theorem \ref{thm:positive Q imply positive P} with the  results established in  \cite{GM, Hang-Yang CPAM ngeq5, Hang-Yang 15 IMRN}, we obtain the following corollary.

\begin{corollary}\label{corollary}
	Let $(M,g)$ be a smooth  compact  Riemannian manifold with  dimension $n\geq 5$ and  $Y(M,[g])>0$. Then, there holds 
	$$Y_4(M,[g])>0\iff \exists \tilde g\in[g]\; \mathrm{with}\;  Q_{\tilde g}\equiv 1 \iff \ker P_g =\{0\}, G_P>0.$$
\end{corollary}

We now outline the structure of this paper. In Section \ref{sec:G-M revisit}, we recall the strong maximum principle and positivity of $P_g$ established by Gursky and Malchiodi. In Section \ref{sec:H-Y revisit}, we recall the remarkable Hang–Yang's  identity and establish a crucial lemma concerning the positivity of the scalar curvature (see Lemma \ref{lem:positive-R_g}). Next, in Section \ref{sec:spectral}, we recall some results on the spectrum of a compact operator established  in \cite{Hang-Yang note} and establish another crucial lemma regarding connections between the spectral radius and the positivity of the $Q$-curvature and scalar curvature (see Lemma \ref{lemm:crucial}). Finally, with these preparations in hand, we complete the proofs of Theorems \ref{thm:main theorem},  \ref{thm:positive Q imply positive P}  and Corollary \ref{corollary} in Section \ref{sec:proof}.

\section{Gursky-Malchiodi's   results revisited}\label{sec:G-M revisit}

As mentioned in the introduction, Gursky and Malchiodi \cite{GM} established the strong maximum principle for $P_g$ by applying the standard maximum principle twice.   The following lemma, which is essentially Theorem 2.2 in \cite{GM}, is restated here for subsequent use. Since their proof is short and elegant, we repeat their proof here.
\begin{lemma}\label{lem:GM}
	Let $(M^n,g)$ be a smooth  compact  Riemannian manifold of dimension $n\geq 5$. Suppose that  there exists a smooth  conformal metric $\tilde g=u^{\frac{4}{n-4}}g$ such that  $Q$-curvature $Q_{\tilde g}\geq 0$, $Q_{\tilde g}\not\equiv 0$, and the scalar curvature $R_{\tilde g}\geq 0$. If $Q_g\geq 0$, then the scalar curvature $R_g>0.$
\end{lemma}
\begin{proof}
	Using \eqref{Q-formula} and $Q_{\tilde g}\ge 0$, we obtain
	\[
	0 \le -\Delta_{\tilde g}J_{\tilde g} + \frac{n}{2}J_{\tilde g}^2 .
	\]
	Applying the strong maximum principle together with $J_{\tilde g}\ge 0$ gives either $J_{\tilde g}>0$ or $J_{\tilde g}\equiv 0$. If $J_{\tilde g}\equiv 0$, then \eqref{Q-formula} yields $Q_{\tilde g}\le 0$, contradicting our assumption that $Q_{\tilde g}\ge 0$ and $Q_{\tilde g}\not\equiv 0$. Hence $J_{\tilde g}>0$  which implies $R_{\tilde g}>0$. Consequently,
	\begin{equation}\label{positiveYamabe}
		Y(M,[g])>0.
	\end{equation}
	
	Define $u_t:=t+(1-t)u$. Clearly, $u_t>0$ for all $t\in[0,1]$, and consider the conformal metrics $g_t=u_t^{\frac{4}{n-4}}g$. For any $t\in[0,1]$, we have
	\[
	Q_{g_t}
	=\frac{2}{n-4}u_t^{-\frac{n+4}{n-4}}P_g(u_t)
	=\frac{2}{n-4}u_t^{-\frac{n+4}{n-4}}
	\left(tP_g(1)+(1-t)P_gu\right).
	\]
	Thus, for all $t\in [0,1]$, one has 
	\begin{equation}\label{eq:Q_g_t}
		Q_{g_t}
		= t\,u_t^{-\frac{n+4}{n-4}}Q_g
		+ (1-t)u_t^{-\frac{n+4}{n-4}}u^{\frac{n+4}{n-4}}Q_{\tilde g}
		\ge 0.
	\end{equation}
	For $t=0$, we have $Q_{g_t}\ge 0$ and $R_{g_t}>0$. Set
	\begin{equation}\label{choice_of_t1}
		t_1 := \sup \left\{ t\in[0,1] \mid \forall s\in[0,t],\ R_{g_s}>0 \right\}.
	\end{equation}
	Since $Q_{g_t}\ge0$ for all $t\in[0,1]$, $R_{g_0}>0$, and $u_t$ is smooth, it follows that $t_1>0$. If $t_1<1$, then by continuity, one has $R_{g_{t_1}}\ge 0$. Moreover, from \eqref{eq:Q_g_t} we have $Q_{g_{t_1}}\ge 0$ and $Q_{g_{t_1}}\not\equiv 0$. Repeating the earlier argument gives $R_{g_{t_1}}>0$. Hence, there exists a small $\epsilon>0$ with $t_1+\epsilon<1$ such that $R_{g_t}>0$ for all $t\in[t_1,t_1+\epsilon]$, contradicting the definition of $t_1$ in \eqref{choice_of_t1}. Therefore $t_1=1$, and so
	$
	R_g = R_{g_1} \ge 0.
$
	Since $Q_g\ge0$, the strong maximum principle implies that either $R_g>0$ or $R_g\equiv0$. If $R_g\equiv0$, then $Y(M,[g])=0$, which contradicts \eqref{positiveYamabe}. This completes the proof.
\end{proof}

The second crucial observation of Gursky and Malchiodi \cite{GM} concerns the positivity of $P_g$ under the assumption \eqref{condition-of-GM}. The following lemma is essentially the same as Proposition 2.3 of \cite{GM}, with the only difference that we also handle the case $Q_g\equiv 0$. We define the first eigenvalue of $P_g$ as follows:
\begin{equation*}
	\lambda_1(P_g):=\inf_{u\in H^2(M)\setminus \{0\}}\frac{\int_MuP_gu\dv}{\int_Mu^2\dv}
\end{equation*}
\begin{lemma}\label{lem:positive-P_g}
		Let $(M^n,g)$ be a smooth  compact  Riemannian manifold of dimension $n\geq 5$. Suppose that the $Q$-curvature  $Q_g\geq 0$ and the scalar curvature $R_g\geq 0$. Then, 
		$\lambda_1(P_g)\geq 0$ with equality holding if and only if $Q_g\equiv 0.$ Moreover, when $\lambda_1(P_g)=0$, one has $\ker P_g=\{constants\}$.
\end{lemma}
\begin{proof}
For any $\varphi\in H^2(M)$,	using \eqref{Paneitz-operator}, one has
	$$\int_M\varphi P_g\varphi\dv=\int_M\left((\Delta_g\varphi)^2-4A_g(\nabla\varphi,\nabla \varphi)+(n-2)J_g|\nabla \varphi|_g^2+\frac{n-4}{2}Q_g\varphi^2\right)\dv.$$
	
	First, for $Q_g\geq 0$ and $R_g\geq 0$, due to the strong maximum principle, one has $R_g>0$ or $R_g\equiv 0$.
	
	When $R_g\equiv 0$, using \eqref{Q-formula} and the assumption $Q_g\geq 0$, there holds $|\Ric_g|^2_g\equiv 0$ and $Q_g\equiv 0$.
	Then, one has 
	$$\int_M\varphi P_g\varphi\dv=\int_M(\Delta_g\varphi)^2\dv\geq 0.$$
It is easy to easy that $\lambda_1(P_g)=0$  and $\ker P_g=\{constants\}$.

Thus, in the subsequent proof, we only need to deal with  the case $Q_g \ge 0$ and $R_g > 0$.

For $n\geq 6$,	applying Bochner formula, one has
$$\int_M(\Delta_g\varphi)^2\dv=\int_M|\nabla^2\varphi|_g^2\dv+\int_M\Ric_g(\nabla\varphi,\nabla\varphi)\dv.$$
	Inserting it into previous identity, there holds
	\begin{align*}
		&\int_M\varphi P_g\varphi\dv\\
		=&\frac{n-6}{n-2}\int_M(\Delta_g\varphi)^2\dv +\frac{4}{n-2}\int_M|\nabla^2\varphi|^2_g\dv\\
		&+\frac{(n-2)^2+4}{n-2}\int_MJ_g|\nabla\varphi|^2_g\dv+\frac{n-4}{2}\int_MQ_g\varphi^2\dv.
	\end{align*}
When $n\geq 6$, $Q_g\geq 0$ and $R_g>0$, it is obvious to see  that $\lambda_1(P_g)\geq 0$. If $\lambda_1(P_g)=0$, then the above identity immediately implies that $Q_g\equiv 0$ and $\ker P_g=\{constants\}$.  Conversely, if $Q_g\equiv 0$, one has  $\ker P_g=\{constants\}$ and then $\lambda_1(P_g)=0$.

For $n=5$, Gursky and Malchiodi \cite{GM} modified the trick in \cite{GM} to obtain that 
$$\int_M\varphi P_g\varphi\dv\geq \frac{3}{7}\int_M(\Delta_g\varphi)^2\dv +\int_MJ_g|\nabla \varphi|^2_g\dv+\frac{1}{2}\int_MQ_g\varphi^2\dv.$$
Thus, $\lambda_1(P_g)\geq 0$. When $\lambda_1(P_g)=0$, one must have $Q_g\equiv 0$ and then $\ker P_g=\{constants\}$. Conversely, if $Q_g\equiv 0$, one has $\ker P_g=\{constants\}$ from \eqref{Paneitz-operator} and then $\lambda_1(P_g)=0$. Thus, we finish the proof.
\end{proof}

\section{Hang-Yang's  identity revisited }\label{sec:H-Y revisit}

In this section, we begin by recalling the remarkable Hang-Yang's identity established in \cite{Hang-Yang 15 IMRN}. For the reader's convenience, we review some necessary background and employ the same notations of \cite{Hang-Yang 15 IMRN}. Since $Y(M,[g])>0$, the conformal Laplacian $L_g$ defined in \eqref{eq:conformal-Laplacian} is positive and hence possesses a symmetric, positive, smooth Green's function $G_L(x,y)$ (see \cite{Aubin}). With the help of the conformal normal coordinate $x_1,\cdots, x_n$ at $p$ (see \cite{LP}), denote $r=|x|$, then one has
\begin{equation}\label{G_L-asymp}
	G_L(p,\cdot)=c_1(n)r^{2-n}(1+O^{(4)}(r))
\end{equation}
where $c_1(n)$ is positive constant depending only on $n$ and  $f=O^{(m)}(r^{\theta})$ means $f\in C^m$ and $\partial_{i_1,\cdots, i_k}f=O(r^{\theta-k})$ for $0\leq k\leq m.$
For a fixed pole $p\in M$, set
\begin{equation*}
	\widehat g_p=G_L(p,\cdot)^{4/(n-2)}g
	\quad\text{on }M\setminus\{p\}.
\end{equation*}
In Proposition 2.1 of  \cite{Hang-Yang 15 IMRN}, Hang and Yang established the following identity
\begin{equation}\label{eq:HY-identity}
	P_g\bigl(G_L(p,\cdot)^{\frac{n-4}{n-2}}\bigr)
	=c_2(n)\delta_p(\cdot)
	-\frac{n-4}{(n-2)^2}G_L(p,\cdot)^{\frac{n-4}{n-2}}
	\left|\Ric_{\widehat g_p}\right|_g^2
\end{equation}
in the distributional sense, where $c_2(n)>0$ is a positive dimensional constant and $\delta_p(\cdot)$ is the Dirac operator.
As pointed out in \cite{Hang-Yang 15 IMRN}, near $p$, one has
\begin{equation}\label{Gamma-asymp}
	G_L(p,\cdot)^{\frac{n-4}{n-2}}
	\left|\Ric_{\widehat g_p}\right|_g^2=O(r^{2-n}).
\end{equation}
Define
\begin{equation}\label{eq:H-def}
	H(p,\cdot):=c_2(n)^{-1}G_L(p,\cdot)^{\frac{n-4}{n-2}}
\end{equation}
and
\begin{equation}\label{eq:Gamma-def}
	\Gamma(p,\cdot)
	:=\frac{n-4}{(n-2)^2c_2(n)}G_L(p,\cdot)^{\frac{n-4}{n-2}}
	\left|\Ric_{\widehat g_p}\right|_g^2.
\end{equation}
Then, in the distributional sense,  there holds
\begin{equation}\label{eq:HY-normalized}
	P_{g}H(p,\cdot)=\delta_p(\cdot)-\Gamma(p,\cdot).
\end{equation}
Obviously, $H(x,y)$ is symmetric. Then, using \eqref{eq:HY-normalized}, $\Gamma(x,y)$ is also symmetric. Meanwhile, both $H(x,y)$ and  $\Gamma(x,y)$ are smooth except the diagonal. 
Then, we define two integral operators by
\begin{equation*}
	T_H(f)(x):=\int_MH(x,y)f(y)\ud\mu_g(y),
\end{equation*}
and
\begin{equation*}
	T_{\Gamma}(f)(x):=\int_M\Gamma(x,y)f(y)\ud\mu_g(y).
\end{equation*}
From \eqref{G_L-asymp} and \eqref{Gamma-asymp}, we know that these two operator are well-defined on $C(M)$.

As we mentioned earlier, the following lemma is crucial to this paper and is inspired by our previous joint work with Xu \cite{Li-Xu}.
\begin{lemma}\label{lem:positive-R_g} 
	Let $(M,g)$   be a smooth  compact  Riemannian manifold with dimension $n\geq 5$ and positive Yamabe invariant.
	Let $f\in C(M)$ satisfy $f\geq 0$ and $f\not\equiv 0$. 
	Then, there holds
$$T_H(f)>0, \quad L_g((T_H(f))^{\frac{n-2}{n-4}})\geq 0.$$
	If $f(x)\geq\varepsilon$ for all $x\in M$ and some positive constant $\varepsilon$, then there holds
	$$
L_g((T_H(f))^{\frac{n-2}{n-4}})>0.
	$$
\end{lemma}

\begin{proof}
	For brevity, set the notations
	$h_y(x):=H(x,y)$ and 
	$U(x):=T_H(f)(x)$.
	From \eqref{eq:H-def}, 
	the function $h_y^{\frac{n-2}{n-4}}$ is a positive constant multiple of
	$G_L(x,y)$, and therefore
	\[
	L_g(h_y^{\frac{n-2}{n-4}}(x))=0, \quad x\in M\setminus\{y\}.
	\]
	Expanding the above  identity and dividing   by $h_y^{\frac{n-2}{n-4}}$ yield
	\begin{equation}\label{eq:h-log-equation}
		\frac{\Delta_g h_y}{h_y}
		=\frac{n-4}{4(n-1)}R_g-\frac{2}{n-4}|\nabla\log h_y|_g^2, \quad x\in M\setminus\{y\}.
	\end{equation}
	For each  fixed point $x$, set the  measure notation
	\begin{equation*}
		d\nu_x(y):=\frac{h_y(x)f(y)}{U(x)}\ud\mu_g(y)
	\end{equation*}
	Using \eqref{G_L-asymp}, one has
	$$h_y(x)=O(r^{4-n}), \quad 
	|\nabla_x\log h_y(x)|_g=O(r^{-1})
$$
near the singular point  $y$.
	Then, using Theorem 6.21 in \cite{LL},   we obtain that 
	\begin{equation*}
		\frac{\nabla U}{U}=\int_M\nabla_x\log h_y(x)\ud\nu_x(y)
	\end{equation*}
	and
	\begin{equation*}
		\frac{\Delta_g U}{U}
		=\frac{n-4}{4(n-1)}R_g-\frac{2}{n-4}\int_M|\nabla_x\log h_y(x)|^2\ud\nu_x(y).
	\end{equation*}
	Consequently,
	\begin{align*}
		\frac{L_g(U^{\frac{n-2}{n-4}})}{U^{\frac{n-2}{n-4}}}
		&=-\frac{4(n-1)}{n-4}\frac{\Delta_g U}{U}
		-\frac{8(n-1)}{(n-4)^2}\frac{|\nabla U|^2}{U^2}+R_g\\
		&=\frac{8(n-1)}{(n-4)^2}\left(
		\int_M|\nabla_x\log h_y(x)|^2\ud\nu_x(y)
		-\left|\int_M\nabla_x\log h_y(x)\ud\nu_x(y)\right|^2
		\right).
	\end{align*}
	With the help of Fubini's theorem, we can rewrite  the last term as follows
	\begin{align*}
		&\int_M|\nabla_x\log h_y(x)|^2\ud\nu_x(y)
		-\left|\int_M\nabla_x\log h_y(x)\ud\nu_x(y)\right|^2\\
		=&\frac{1}{2}U^{-2}\iint_{M\times M}
		h_y(x)h_z(x)f(y)f(z)|\nabla_x\log h_y(x)-\nabla_x\log h_z(x)|^2
		\ud\mu_g(y)\ud\mu_g(z).
	\end{align*}
	On one hand, using the above identity, obviously, one has  $L_g(U^{\frac{n-2}{n-4}})\geq 0.$
	On the other hand, if  $f\geq\varepsilon>0$, using the  identity again, it is easy to show that  the above identity is strictly positive. Then, one has $L_g(U^{\frac{n-2}{n-4}})>0.$
	Thus, we finish the proof.
\end{proof}

\section{Spectrum of $T_\Gamma$ }\label{sec:spectral}
In this section, we will recall several results of Hang and Yang from \cite{Hang-Yang note} concerning the spectrum of the operator $T_\Gamma$.
The first proposition  essentially has been established in \cite{Hang-Yang note, Hang-Yang 15 IMRN}. For the reader's convenience, we state as follows and sketch the proof.
\begin{prop}\label{prop:operator-identity} 	Let $(M,g)$ of  dimension $n\geq 5$ be a smooth  compact  Riemannian manifold which is not conformally equivalent to standard sphere.
	For each $f\in C(M)$, there holds
	\begin{equation}\label{eq:P-T-H}
		P_g(T_H(f))=f-T_\Gamma(f)
	\end{equation}
	in the distributional sense.  Moreover, $T_\Gamma:C(M)\to C(M)$ is a positive
	compact operator and  the spectral radius $r_\sigma(T_\Gamma)>0$.
\end{prop}

\begin{proof}
For any smooth $\varphi$, using Fubini's theorem and Hang-Yang's identity \eqref{eq:HY-normalized}, there holds
\begin{align*}
	&\int_MP_g(\varphi(y))T_H(f)(y)\ud\mu_g(y)\\
	=&\int_Mf(x)\int_MP_g(\varphi(y))H(x,y)\ud\mu_g(y)\ud\mu_g(x)\\
	=&\int_Mf(x)\varphi(x)\ud\mu_g(x)-\int_Mf(x)T_{\Gamma}(\varphi)(x)\ud\mu_g(x)\\
	=&\int_M(f(x)-T_{\Gamma}(f)(x))\varphi(x)\ud\mu_g(x)
\end{align*}
which proves \eqref{eq:P-T-H}.

	 From the definition \eqref{eq:Gamma-def}, one has $\Gamma(x,y)\geq 0$.
	 It is not hard to see $\int_M\Gamma(p,y)\ud\mu_g(y)$ is continuous with respect to $p$. Thus, we set
	 $$c:=\min_{p\in M}\int_M\Gamma(p,y)\ud\mu_g(y).$$
	 First, by \eqref{Gamma-asymp},   $c$ must be  finite and nonnegative. Now,
	 we claim that $c>0$ by following the argument of the proof of Lemma 3.1 in \cite{Hang-Yang 15 IMRN}. 
	 We argue by contraction and assume that $c=0$. Then, for some $p\in M$, one has
	 $|\Ric_{\widehat g_p}|_g=0$ on $M\setminus\{p\}$. Since $(M\setminus\{p\},\widehat g_p)$ is asymptotically flat, it follows from the 
	 relative volume comparison theorem that $(M\setminus\{p\},\widehat g_p)$  is isometric to the standard $\R^n$.
 In particular,  $(M^n, g)$ must be locally conformally flat and simply connected compact
manifold, hence, it is conformal to the standard $\ms^n$ by \cite{Kuiper} which contradicts our assumption.
	
	Since $$\|T_\Gamma(\varphi)\|_{L^\infty(M)}\leq \left(\max_{p\in M}\int_M\Gamma(p,y)\ud\mu_g(y)\right)\|\varphi\|_{L^\infty(M)},$$
 one has
$  \|T_\Gamma\|<+\infty.$

Using \eqref{Gamma-asymp}, the  standard singular 
 integral  estimates imply  that $T_\Gamma$ maps bounded
sets in $C(M)$ into an equicontinuous family; hence it is compact by the
Arzel\`a-Ascoli theorem.

Meanwhile, one has 
$$\|T_\Gamma^k\|\geq \| T_\Gamma^k(1)\|_\infty\geq c^k.$$
With	 the help of Gelfand's formula (see \cite{Lax}), one has
$$r_\sigma(T_\Gamma)=\lim_{k\to\infty}\|T_\Gamma^k\|^{\frac{1}{k}}\geq c>0.$$
Thus, we finish the proof.
\end{proof}

The final crucial step during  the proof of Theorems \ref{thm:main theorem} and Theorem \ref{thm:positive Q imply positive P}  is inspired by Hang and Yang. In \cite{Hang-Yang note}, they established a key result relating the range of the spectral radius to the existence of a conformal metric with positive $Q$-curvature. For the reader's convenience, we repeat the statement of their result as follows for subsequent use. \begin{theorem}\label{thm:Hang-Yang theorem}(Theorem 3.5 in \cite{Hang-Yang note})
Let $(M,g)$  be a smooth  compact  Riemannian manifold with dimension $n\geq 5$ and  positive Yamabe invariant $Y(M,[g])$.  Then,
	the existence of a conformal metric $\tilde g\in[g]$ with $Q_{\tilde g}>0$ is equivalent to $r_{\sigma}(T_\Gamma)<1.$
\end{theorem}

In contrast to  Theorem \ref{thm:Hang-Yang theorem}, we further  prove  that,  if $r_\sigma(T_\Gamma)<1$, then there exists  a  conformal metric with both positive  $Q$-curvature and positive scalar curvature.

\begin{lemma}\label{lemm:crucial}
	Let $(M,g)$  be a smooth  compact  Riemannian manifold with dimension $n\geq 5$ and  positive Yamabe invariant $Y(M,[g])$.  If $r_\sigma(T_\Gamma)<1$, then there exists a conformal metric $\tilde g\in [g]$ such that the $Q$-curvatue $Q_{\tilde g}\geq 0$,  $Q_{\tilde g}\not \equiv 0$, and the scalar curvature $R_{\tilde g}>0$.
\end{lemma}
\begin{proof}
	
	When $(M,g)$ is conformally equivalent  to standard sphere, Theorem \ref{thm:Hang-Yang theorem} shows that there exists a conformal metric $\tilde g\in [g]$  with positive $Q$-curvature $Q_{\tilde g}$. Then, the scalar curvature $R_{\tilde g}>0$ follows from Theorem 1.1 in \cite{Li-Xu} by  applying  stereographic projection.
	
In the remaining case, based on Proposition \ref{prop:operator-identity},  Krein-Rutman theorem (see \cite{Lax})  shows that there exists a  non-zero   function
$f\in C(M)$ such that
\begin{equation*}
	f\geq 0, \quad T_{\Gamma}(f)=r_{\sigma}(T_{\Gamma})f.
\end{equation*}
Since $f$ is an eigenfunction, with the help of bootstrap argument and standard regularity theory, one has $f\in C^\infty(M)$ and then so is  $T_H(f)$. Due to Lemma \ref{lem:positive-R_g}, one has $T_H(f)>0.$ 
Using  Proposition
\ref{prop:operator-identity} and our assumption $r_\sigma(T_\Gamma)<1$,  there holds
\begin{equation}\label{Q_tilde g geq 0}
	P_g(T_H(f))=(1-r_\sigma(T_\Gamma))f\geq 0, \quad P_g(T_H(f))\not \equiv 0.
\end{equation}
Choose the conformal metric $\tilde g:=(T_H(f))^{\frac{4}{n-4}}g$. Using \eqref{Q_tilde g geq 0} and \eqref{Q-def}, one has  $Q_{\tilde g}\geq 0$ and $Q_{\tilde g}\not \equiv 0$.
	   Lemma \ref{lem:positive-R_g} yields that  scalar curvature $R_{ \tilde g}$  satisfies
	$$R_{\tilde g}=T_H(f)^{-\frac{n+2}{n-4}}L_g\left((T_H(f))^{\frac{n-2}{n-4}}\right)\geq 0.$$
	Then, using Lemma \ref{lem:GM}, one has $R_{\tilde g}>0.$
	Thus, we finish the proof.
\end{proof}

\section{Proof and further discussion}\label{sec:proof}

{\bf Proof of Theorem \ref{thm:main theorem}:}

In fact, it suffices to treat the case where  $(M,g)$ is not conformally equivalent to standard sphere, as the remaining case is trivial.
First, we claim  that 
	\begin{equation}\label{r_sigam<1}
		r_{\sigma}(T_{\Gamma})<1.
	\end{equation}
We argue by contradiction and  suppose that $r_{\sigma}(T_{\Gamma})\geq1.$ 
 Follow the same argument and notations as in the proof of Lemma \ref{lemm:crucial}.
Using Proposition \ref{prop:operator-identity} and Krein-Rutman theorem, there exists  a non-zero  function
$f\in C(M)$ such that
\begin{equation}
	f\geq 0,\quad  T_{\Gamma}(f)=r_{\sigma}(T_{\Gamma})f.
\end{equation}
By regularity theory,  $f$ is smooth. Then, $T_H(f)$ is a positive and smooth function.
 Using Proposition
	\ref{prop:operator-identity},  there holds
	\begin{equation}\label{eq:Pu-eigen}
		P_g(T_H(f))=(1-r_\sigma(T_\Gamma))f.
	\end{equation}
	With the help of the above identity \eqref{eq:Pu-eigen} and the assumption $r_\sigma(T_\Gamma)\geq 1$, the Paneitz energy  satisfies
	\begin{equation}\label{eq:energy-nonpositive}
\int_MT_H(f)P_g(T_H(f))\dv
		=(1-r_\sigma(T_\Gamma))\int_MfT_H(f)\dv\leq0.
	\end{equation}
Now, for small positive constant $\varepsilon$,  we consider 
$$u_\epsilon:=T_H(f+\varepsilon).$$
Consider the conformal metric $g_{\varepsilon}=u_\varepsilon^{\frac{4}{n-4}}g$.
Using Lemma \ref{lem:positive-R_g}, the scalar curvature $R_{g_\varepsilon}$ satisfies
\begin{equation}\label{R_epsilon-positive}
	R_{g_\varepsilon}=u_\varepsilon^{-\frac{n+2}{n-4}}L_g(u_\varepsilon^{\frac{n-2}{n-4}})>0.
\end{equation} 
Then, using \eqref{R_epsilon-positive} and  the definition \eqref{eq:Y4-star}, there holds
	\begin{equation}\label{eq:Y4star-lower-bound}
	\int_Mu_\varepsilon P_g u_\varepsilon\dv
		\geq Y_4^*(M,[g])
		\left(\int_Mu_\varepsilon^{\frac{2n}{n-4}}\dv\right)^{(n-4)/n}.
	\end{equation}	
On one hand, letting $\varepsilon\to 0$ and using \eqref{eq:energy-nonpositive}, one has 
	$$
	\int_Mu_\varepsilon P_g u_\varepsilon\dv\to \int_MT_H(f)P_g(T_H(f))\dv\leq0.
$$
On the other hand,  letting $\varepsilon\to 0$, one has 
$$
	Y_4^*(M,[g])
	\left(\int_Mu_\varepsilon^{\frac{2n}{n-4}}\dv\right)^{\frac{n-4}{n}}
	\to 
	Y_4^*(M,[g])
	\left(\int_M(T_H(f))^{\frac{2n}{n-4}}\dv\right)^{\frac{n-4}{n}}>0.
$$
which contradicts \eqref{eq:Y4star-lower-bound}.  Therefore,
the claim \eqref{r_sigam<1} holds.  Finally, combining with Lemma \ref{lemm:crucial} and Theorem D in \cite{GM},  we finish the proof.  \qed

\vspace{3em}

{\bf Proof of Theorem \ref{thm:positive Q imply positive P}:}

First, by Theorem 1.1 in \cite{Hang-Yang CPAM ngeq5}, there exists a conformal metric with positive constant $Q$-curvature. Applying Theorem \ref{thm:Hang-Yang theorem}, we obtain that  $r_{\sigma}(T_\Gamma)<1$.

Lemma \ref{lemm:crucial} ensures the existence of a conformal metric $\tilde g=u^{\frac{4}{n-4}}g$ such that $Q_{\tilde g}\ge 0$, $Q_{\tilde g}\not\equiv 0$ and $R_{\tilde g}>0$. Then,  Lemma \ref{lem:GM}  implies $R_g>0$.
Finally, Lemma \ref{lem:positive-P_g} implies that $P_g$ is also positive, which completes the proof. \qed

\vspace{2em}

{\bf Proof of Corollary \ref{corollary}:}

{\bf Case 1:} $Y_4(M,[g])\Rightarrow \exists \tilde g\in [g], Q_{\tilde{g}}=1.$

Using Theorem \ref{thm:main theorem} and \eqref{Y_4-relation}, there exists a conformal metric with positive $Q$-curvature and positive scalar curvature. Then, applying Theorem D in \cite{GM},  we obtain a conformal metric with $Q$-curvature equal to one.

{\bf Case 2:} $ \exists \tilde g\in [g], Q_{\tilde{g}}=1\Rightarrow Y_4(M,[g])>0$.

This directly follows from Theorem \ref{thm:positive Q imply positive P}.

{\bf Case 3:} $\exists \tilde g\in [g], Q_{\tilde{g}}=1\Rightarrow \ker P_g=\{0\}, G_P>0$.
This directly follows from Theorem 1. 1 in \cite{Hang-Yang 15 IMRN}.

{\bf Case 4:} $\ker P_g=\{0\}, G_P>0 \Rightarrow \exists \tilde g\in [g], Q_{\tilde{g}}=1$.
It follows from Theorem 1.1 in \cite{Hang-Yang 15 IMRN} and Theorem 1.1 in \cite{Hang-Yang CPAM ngeq5}. \qed

\vspace{3em}

Combining Theorem \ref{thm:main theorem} with Theorem \ref{thm:positive Q imply positive P}, we remove the assumption $Q_g\geq 0$ and $Q_g\not \equiv 0$ in Theorem 1.3 of  \cite{Hang-Yang CPAM ngeq5} and then the theorem can be stated as follows.

\begin{theorem}
		Let $(M,g)$ be a smooth  compact  Riemannian manifold of dimension $n\geq 5$.  Suppose that
		\begin{equation}\label{positive-invariant-condition}
			Y(M,g)>0, \quad  Y_4(M,[g])>0.
		\end{equation} Then, there holds
		\begin{enumerate}
			\item $Y_4(M,[g])\leq Y_4(\ms^n,[g_0])$ with equality holding if and only if $(M,g)$ is conformally equivalent to standard sphere.
			\item $Y_4(M,[g])$ is always achieved. Any minimizer must be smooth and cannot
			change sign.
		\end{enumerate}		
\end{theorem}

\end{document}